\documentclass[11pt]{article}
\usepackage{amsmath, amsfonts, amssymb, amsthm}

\newtheorem{thm}{Theorem}
\newtheorem{prop}[thm]{Proposition}
\newtheorem{lem}[thm]{Lemma}

\newtheorem{cor}[thm]{Corollary}
\newtheorem{rem}[thm]{Remark}
\newtheorem{df}[thm]{Definition}

\renewcommand{\epsilon}{\varepsilon}
\renewcommand{\phi}{\varphi}

\newcommand{\BB}{\mathbb}

\newcommand{\im}{\operatorname{Im}}
\newcommand{\re}{\operatorname{Re}}

\newcommand{\parag}{${\cal x}$}
\newcommand{\rot}{{\cal R}}

\begin{document}

\title{On a Quaternionic Analogue of the Cross-Ratio}
%and Some Properties of Quaternionic Fractional Linear Transformations}
\author{Ewain Gwynne\footnote{Undergraduate student at Northwestern University,
Evanston, IL 60208} 
 and Matvei Libine\footnote{Department of Mathematics, Indiana University,
Rawles Hall, 831 East 3rd St, Bloomington, IN 47405}}
\maketitle

\begin{abstract}
In this article we study an exact analogue of the cross-ratio for the algebra
of quaternions $\BB H$ and use it to derive several interesting properties of
quaternionic fractional linear transformations.
In particular, we show that there exists a fractional linear transformation
$T$ on $\BB H$ mapping four distinct quaternions $q_1$, $q_2$, $q_3$ and $q_4$
into $q'_1$, $q'_2$, $q'_3$ and $q'_4$ respectively if and only if the
quadruples $(q_1, q_2, q_3, q_4)$ and $(q'_1, q'_2, q'_3, q'_4)$ have the same
cross-ratio.
If such a fractional linear transformation $T$ exists it is never unique.
However, we prove that a fractional linear transformation on $\BB H$ is uniquely
determined by specifying its values at five points in general position.
We also prove some properties of the cross-ratio including criteria
for four quaternions to lie on a single circle (or a line) and
for five quaternions to lie on a single 2-sphere (or a 2-plane).
As an application of the cross-ratio, we prove that fractional linear
transformations on $\BB H$ map spheres (or affine subspaces) of dimension
1, 2 and 3 into spheres (or affine subspaces) of the same dimension.
\end{abstract}


\section{Introduction}

In this article we study an exact analogue of the cross-ratio for the algebra
of quaternions $\BB H$ and use it to derive several interesting properties of
quaternionic fractional linear transformations.
(Note that some authors prefer to call fractional linear transformations
``M\"obius transformations''.)
Recall that, if $z_1, z_2, z_3, z_4 \in \BB C$ are
distinct complex numbers, then their cross-ratio is
\begin{equation}  \label{C-ratio}
R_{\BB C}(z_1, z_2, z_3, z_4) = \frac{(z_4-z_1)(z_2-z_3)}{(z_2-z_1)(z_4-z_3)}.
\end{equation}
The cross-ratio has many important properties including invariance under
fractional linear transformations:
$$
R_{\BB C}(z_1, z_2, z_3, z_4) =
R_{\BB C} \bigl( T(z_1), T(z_2), T(z_3), T(z_4) \bigr)
$$
for all fractional linear transformations $T$ on $\BB C$.
Another way to characterize the cross-ratio is as follows.
Let $T$ be the unique fractional linear transformation sending $z_1$, $z_2$ and
$z_3$ into $0$, $1$ and $\infty$ respectively, then
$R_{\BB C}(z_1, z_2, z_3, z_4) = T(z_4)$.
The cross-ratio can be used to determine if given four distinct points
$z_1, z_2, z_3, z_4 \in \BB C$ lie on a circle or a straight line.
This happens if and only if $R_{\BB C}(z_1, z_2, z_3, z_4)$ is real.
As a consequence, one immediately obtains that a fractional linear
transformation maps circles and lines into circles and lines.
These properties are discussed in detail in most
complex analysis textbooks including, for example, \cite{A}.

In Section 2 we introduce our notations and list some basic facts about
quaternions.

In Section 3 we state the definition of the quaternionic cross-ratio
(Definition \ref{def}); it is essentially the same as the one introduced
in \cite{BP}, \cite{HJHP} and \cite{HJ}.
%(The earlier article \cite{BP} defines
%the cross-ratio for purely imaginary quaternions only, then
%the book \cite{HJ} extends this definition to all quaternions.)
We prove that there exists a fractional linear transformation $T$ on $\BB H$
mapping four distinct points $q_1$, $q_2$, $q_3$ and $q_4$ into
$q'_1$, $q'_2$, $q'_3$ and $q'_4$ respectively if and only if the quadruples
$(q_1, q_2, q_3, q_4)$ and $(q'_1, q'_2, q'_3, q'_4)$ have cross-ratios with
the same real parts and norms (Theorem \ref{main}).
The fact that the fractional linear transformations preserve the real part
and the norm of the quaternionic cross-ratio was proved in \cite{HJ}
by a different method.
The invariants of fractional linear transformations discussed in
Remark \ref{igor} were suggested to us by Igor Frenkel
(see also Proposition 3 in \cite{C1}).
We also note that a fractional linear transformation on $\BB H$ is never
determined uniquely by specifying its values at four points.

In Section 4 we derive various properties of the cross-ratio and
fractional linear transformations on $\BB H$.
First, we fix the images of three points in $\BB H$ and show that the set
of all possible images of a fourth point under a fractional linear
transformation is either a 2-sphere, a 2-plane or a single point
(Proposition \ref{4points}).
As a consequence of this result, we show that fractional linear
transformations on $\BB H$ map spheres (or affine subspaces) of dimension
1, 2 and 3 into spheres (or affine subspaces) of the same dimension
(Theorem \ref{spheres-thm}).
Theorem \ref{spheres-thm} was originally proved in \cite{BG}, but their proof
did not use the quaternionic cross-ratio.
Then we show that four quaternions $q_1, q_2, q_3, q_4$ lie on a single circle
(or a line) if and only if their cross-ratio is real (Proposition \ref{real}).
The last property is not new, see, for example,
\cite{BP}, \cite{HJHP}, \cite{HJ}, \cite{BG}.

In Section 5 we give a necessary and sufficient condition in terms of
cross-ratios for the existence of a fractional linear transformation with
prescribed values at five different points (Proposition \ref{main2}).
We prove that such a fractional linear transformation is uniquely
determined if these five points do not lie on a single 2-sphere or a 2-plane
(Proposition \ref{unique}).
Then we give a criterion for five quaternions to lie on a single 2-sphere
(or a 2-plane) (Lemma \ref{sphere-crit}).
Finally, we fix the images of four points in $\BB H$ and show that the set
of all possible images of a fifth point under a fractional linear
transformation is either a circle, a line or a single point
(Proposition \ref{5points}).
To the best of authors knowledge, the results of this section are new.

Finally, we comment that other properties of the quaternionic cross-ratio are
discussed in the book \cite{HJ}.
Moreover, there is a Clifford algebra  analogue of the cross-ratio
(for example, see \cite{C1}, \cite{BHJ}, \cite{HJ}),
and it would be interesting to determine if the results of our article extend
to Clifford algebras. For example, Proposition 1 in \cite{C1}
is a Clifford algebra analogue of Proposition \ref{real}.
We expect most results stated in our paper to have such an extension.

This paper was written as a part of REU (research experiences for
undergraduates) project at Indiana University during Summer 2011.
We would like to thank Professors Kevin Pilgrim and Bruce Solomon,
and the NSF for providing the organization and funding for the REU program
that made this project possible.
This REU program was supported by the NSF grant DMS-0851852.
The second author was supported by the NSF grant DMS-0904612.

\section{Preliminaries}

In this section we introduce our notations and list basic facts about
quaternions that we will use.
There are many texts providing elementary introductions to quaternions
including, for example, \cite{BFLPP} and \cite{HJ}.
Recall that the quaternions $\BB H$ form an algebra over $\BB R$ generated by
the units $1$, $i$, $j$, $k$.
The multiplicative structure is determined by the rules
\begin{center}
\begin{tabular}{c}
$1q=q1=q, \qquad \forall q \in \BB H,$ \\
$ij=-ji, \qquad ik=-ki, \qquad jk=-kj,$ \\
$i^2=j^2=k^2=ijk=-1,$
\end{tabular}
\end{center}
%$$
%1q=q1=q, \qquad \forall q \in \BB H,
%$$
%$$
%ij=-ji, \qquad ik=-ki, \qquad jk=-kj,
%$$
%$$
%i^2=j^2=k^2=ijk=-1,
%$$
and the fact that $\BB H$ is a division ring.
We write elements $q \in \BB H$ as
$$
q = t + ix + jy + kz, \qquad t,x,y,z \in \BB R,
$$
and use notations for
\begin{center}
\begin{tabular}{lcl}
quaternionic conjugate: & \qquad & $\bar q = t - ix - jy - kz$, \\
norm: & \qquad & $|q| = \sqrt{q \bar q} = \sqrt{\bar q q}
= \sqrt{t^2+x^2+y^2+z^2} \quad \in \BB R$, \\
real part: & \qquad & $\re q = (q+\bar q)/2 = t \quad \in \BB R$, \\
imaginary part: & \qquad &
$\im q = (q-\bar q)/2 = ix + jy + kz \quad \in \BB H$, \\
non-zero quaternions: & \qquad & $\BB H^{\times} = \BB H \setminus \{0\}$.
\end{tabular}
\end{center}
%\begin{align*}
%\bar q &= t - ix - jy - kz, \\
%|q| &= \sqrt{q \bar q} = \sqrt{\bar q q} = \sqrt{t^2+x^2+y^2+z^2}
%\quad \in \BB R, \\
%\re q &= \frac12(q+\bar q) = t \quad \in \BB R, \\
%\im q &= \frac12(q-\bar q) = ix + jy + kz \quad \in \BB H.
%\end{align*}
Then we have
\begin{align*}
\overline{q_1q_2} &= \bar q_2 \cdot \bar q_1, \\
|q_1q_2| &= |q_1| \cdot |q_2|, \\
\re(q_1q_2) &= \re(q_2q_1), \\
q^{-1} & = \bar q / |q|^2.
\end{align*}

Occasionally, it is convenient to use a matrix realization of quaternions.
We realize $\BB H$ as a subalgebra of the algebra of $2 \times 2$
complex matrices by identifying the units $1,i,j,k \in \BB H$ with
$$
1 \leftrightsquigarrow \begin{pmatrix} 1 & 0 \\ 0 & 1 \end{pmatrix}, \qquad
i \leftrightsquigarrow \begin{pmatrix} 0 & -i \\ -i & 0 \end{pmatrix}, \qquad
j \leftrightsquigarrow \begin{pmatrix} 0 & -1 \\ 1 & 0 \end{pmatrix}, \qquad
k \leftrightsquigarrow \begin{pmatrix} -i & 0 \\ 0 & i \end{pmatrix}.
$$
Thus
$$
q = t+ix+jy+kz \quad \leftrightsquigarrow \quad
\begin{pmatrix} t-iz & -y-ix \\ y-ix & t+iz \end{pmatrix},
%= \begin{pmatrix} q_{11} & q_{12} \\ q_{21} & q_{22} \end{pmatrix},
$$
and we get a matrix realization of quaternions:
$$
\BB H \:\simeq\: \biggl\{
\begin{pmatrix} a & b \\ -\overline{b} & \overline{a} \end{pmatrix}
\in \mathfrak{gl}(2,\BB C);\: a,b \in \BB C \biggr\}.
$$
Under this identification, $|q|^2$ and $2 \re q$ are respectively
the determinant and trace of the corresponding matrix, and
the unit sphere in $\BB H$ gets identified with $SU(2)$ in $GL(2,\BB C)$:
$$
S^3 = \{ q \in \BB H ;\: |q|=1 \} \simeq
\biggl\{ \begin{pmatrix} a & b \\ -\overline{b} & \overline{a} \end{pmatrix}
\in GL(2,\BB C);\: a,b \in \BB C ,\: \det
\begin{pmatrix} a & b \\ -\overline{b} & \overline{a} \end{pmatrix} =1 \biggr\}
= SU(2).
$$

Each non-zero quaternion $a \in \BB H^{\times}$ induces a transformation
$\operatorname{Conj}_a$ on $\BB H$:
$$
\operatorname{Conj}_a: q \mapsto aqa^{-1}, \qquad \forall q \in \BB H.
$$
The transformation $\operatorname{Conj}_a$ preserves the real parts and norms,
hence determines a rotation in the 3-dimensional space consisting of purely
imaginary quaternions, i.e. an element of $O(3)$.
Because $\BB H^{\times}$ is connected, $\operatorname{Conj}_a \in SO(3)$
and we obtain a map $\operatorname{Conj}: \BB H^{\times} \to SO(3)$,
$a \mapsto \operatorname{Conj}_a$.
Since $\operatorname{Conj}_a$ is the identity transformation whenever
$a \in \BB R^{\times} = \BB R \setminus \{0\}$, $\operatorname{Conj}$
descends to a map
$$
\BB H^{\times}/\BB R^{\times} = SU(2)/\{\pm Id\} \to SO(3). 
$$
It is well-known that this map is an analytic isomorphism.
(See, for example, \cite{H}, Example II in Chapter V, \parag 2.)
Note that if $\im a \ne 0$, then the axis of the rotation
$\operatorname{Conj}_a$ is the line passing through $\im a$.
These two lemmas follow immediately from the above discussion:

\begin{lem}  \label{conj}
Let $q, q' \in \BB H$, then $q' = aqa^{-1}$ for some
$a \in \BB H^{\times}$ if and only if
$|q|=|q'|$ and $\re q = \re q'$.
\end{lem}

\begin{lem}  \label{commute}
Two quaternions $q_1, q_2 \in \BB H$ commute with each other if and only if
one of $\im q_1$, $\im q_2$ is a real multiple of the other.
\end{lem}

Let $GL(2,\BB H)$ be the group consisting of invertible $2 \times 2$
matrices with entries in $\BB H$.
It acts on $\widehat{\BB H} = \BB H \cup \{\infty\}$
by fractional linear (or M\"obius) transformations:
$$
\pi(\gamma): \: q \: \mapsto \: (aq+b)(cq+d)^{-1},
%= (a'-qc')^{-1}(-b'+qd'), \\
\qquad \text{where }
\gamma = \begin{pmatrix} a & b \\ c & d \end{pmatrix} \in GL(2,\BB H).
%\gamma^{-1} = \begin{pmatrix} a' & b' \\ c' & d' \end{pmatrix}.
$$

First we show that there always exists a fractional linear transformation
sending any three distinct points $q_1, q_2, q_3$ in $\widehat{\BB H}$
into 0, 1 and $\infty$ respectively.

\begin{lem}  \label{01infty}
Given any three distinct points $q_1, q_2, q_3 \in \widehat{\BB H}$,
there exists a fractional linear transformation $T_{q_1,q_2,q_3}$ such that
$T_{q_1,q_2,q_3}(q_1)=0$, $T_{q_1,q_2,q_3}(q_2)=1$ and $T_{q_1,q_2,q_3}(q_3)=\infty$.
\end{lem}

\begin{proof}
Applying a fractional linear transformation if necessary, we can assume that
none of $q_1, q_2, q_3$ is $\infty$. Let
$$
T_{q_1,q_2,q_3}(q) = (q_2 - q_1)^{-1} (q - q_1)(q - q_3)^{-1}(q_2 - q_3).
$$
Then $T_{q_1,q_2,q_3}$ maps $q_1$ to $0$, $q_2$ to $1$, $q_3$ to $\infty$, and
$T_{q_1,q_2,q_3} = \pi(\gamma)$, where
$$
\gamma = \begin{pmatrix} (q_2 - q_1)^{-1} & 0 \\ 0 &  (q_2-q_3)^{-1} \end{pmatrix}
\begin{pmatrix} 1 & -q_1 \\ 1 & -q_3 \end{pmatrix} \in GL(2,\BB H).
$$
\end{proof}

We will often use this lemma to reduce the general case when $q_1,q_2,q_3$
are three distinct points in $\widehat{\BB H}$ to the case when
$q_1=0$, $q_2=1$, $q_3=\infty$.

\begin{cor}
Given any three distinct points $q_1, q_2, q_3 \in \widehat{\BB H}$
and another triple of distinct points $q'_1, q'_2, q'_3 \in \widehat{\BB H}$,
there exists a fractional linear transformation $T$ such that
$T(q_n)=q'_n$, $n=1,2,3$.
\end{cor}

\begin{proof}
Setting $T= (T_{q'_1,q'_2,q'_3})^{-1} \circ T_{q_1,q_2,q_3}$
gives the desired transformation.
\end{proof}

The following lemma implies that the fractional linear transformation
in the above corollary is never unique.

\begin{lem}  \label{conj-transf}
A fractional linear transformation $T$ maps $0$ to $0$, $1$ to $1$,
and $\infty$ to $\infty$ if and only if $T$ is of the form
$T(q) = a q a^{-1}$ for some $a \in \BB H^{\times}$. 
\end{lem}

\begin{proof}
Write $T(q) = (aq + b)(cq +d)^{-1}$ for some $a, b , c, d \in \BB H$.
The condition that $T(0)=0$ implies that $b = 0$;
the condition that $T(\infty)=\infty$ implies that $c = 0$;
the condition that $T(1)=1$ implies that $a+b=c+d$, and hence that $a=d$.
Thus, $T(q) = aqa^{-1}$, as desired. 
\end{proof}

We finish this section by restating Lemma 10 from \cite{FL1}:

\begin{lem}  \label{Lemma 10}
Let $q_1, q_2 \in \BB H$,
$\gamma = \begin{pmatrix} a & b \\ c & d \end{pmatrix} \in GL(2,\BB H)$.
Write $\gamma^{-1} = \begin{pmatrix} a' & b' \\ c' & d' \end{pmatrix}$,
$\tilde q_1 = (aq_1+b)(cq_1+d)^{-1}$ and $\tilde q_2 = (aq_2+b)(cq_2+d)^{-1}$.
Then
\begin{align*}
(\tilde q_1 - \tilde q_2) &=
(a'-q_2c')^{-1} \cdot (q_1-q_2) \cdot (cq_1+d)^{-1}  \\
&= (a'-q_1c')^{-1} \cdot (q_1-q_2) \cdot (cq_2+d)^{-1}.
\end{align*}
\end{lem}

\section{Quaternionic Cross-Ratio}

In this section we introduce a quaternionic analogue of the cross-ratio
and prove some of its properties.

\begin{df}  \label{def}
Given four distinct points $q_1, q_2, q_3, q_4 \in \widehat{\BB H}$,
we define their {\em cross-ratio} to be
$$
Q(q_1, q_2, q_3, q_4) =
(q_2 - q_1)^{-1} (q_4 - q_1)(q_4 - q_3)^{-1}(q_2 - q_3) \quad \in \BB H.
$$
(If one of the $q_n=\infty$, the cross-ratio is defined by letting this
$q_n \to \infty$ and taking limits.)
\end{df}

Setting $q_1=0$, $q_2=1$ and letting $q_3 \to \infty$ we get:
\begin{equation}  \label{q}
Q(0, 1, \infty, q) = q.
\end{equation}
It will be convenient to introduce a notation
$$
R_{\BB H}(q_1, q_2, q_3, q_4) = \bigl(
|Q(q_1, q_2, q_3, q_4)|, \re Q(q_1, q_2, q_3, q_4) \bigr) \quad \in \BB R^2.
$$
Comparing the definition of $R_{\BB H}$ with the complex case (\ref{C-ratio}),
note that  the magnitude of a complex number together with its real part
determine that complex number up to complex conjugation,
but this is not the case for quaternions.

This definition originally appeared in \cite{BP}, where the cross-ratio
was defined for purely imaginary quaternions only,
then in \cite{HJHP} and \cite{HJ} it was extended to all quaternions.
Note that some authors define the quaternionic cross-ratio as
$$
\re Q(q_1, q_2, q_3, q_4) \pm i |\im Q(q_1, q_2, q_3, q_4)| \quad \in \BB C.
$$
This is a complex number defined up to conjugation which uniquely determines
$R_{\BB H}(q_1, q_2, q_3, q_4)$ and in turn can be recovered from
$R_{\BB H}(q_1, q_2, q_3, q_4)$.

Perhaps, the most important property of the quaternionic cross-ratio is that
the quantity $R_{\BB H}(q_1, q_2, q_3, q_4)$ stays invariant under fractional
linear transformations. This property is a part of the following theorem.

\begin{thm}  \label{main}
Given four distinct points $q_1, q_2, q_3, q_4 \in \widehat{\BB H}$
and another quadruple of distinct points
$q'_1, q'_2, q'_3, q'_4 \in \widehat{\BB H}$, there exists a fractional linear
transformation $T$ such that $T(q_n)=q'_n$, $n=1,2,3,4$, if and only if
$$
R_{\BB H}(q_1, q_2, q_3, q_4) = R_{\BB H}(q'_1, q'_2, q'_3, q'_4).
$$
\end{thm}

\begin{proof}
First we show that $R_{\BB H}(q_1, q_2, q_3, q_4)$ is invariant under fractional
linear transformations.
Pick any $\gamma= \begin{pmatrix} a & b \\ c & d \end{pmatrix} \in GL(2,\BB H)$,
let $T(q)=(aq+b)(cq+d)^{-1}$ be the corresponding fractional linear
transformation and write
$\gamma^{-1} = \begin{pmatrix} a' & b' \\ c' & d' \end{pmatrix}$.
By Lemma \ref{Lemma 10},
\begin{multline} \label{Q-conj}
Q \bigl( T(q_1), T(q_2), T(q_3), T(q_4) \bigr) \\
%= (cq_2+d)(q_2 - q_1)^{-1} (a'-q_1c')  (a'-q_1c')^{-1}(q_4 - q_1)(cq_4+d)^{-1}
% (cq_4+d)(q_4 - q_3)^{-1} (a'-q_3c')  (a'-q_2c')^{-1} (q_2 - q_3)(cq_1+d)^{-1}
% \\
= (cq_2+d) (q_2 - q_1)^{-1} (q_4 - q_1)(q_4 - q_3)^{-1}(q_2 - q_3)
(cq_2+d)^{-1} \\
= (cq_2+d) Q(q_1, q_2, q_3, q_4) (cq_2+d)^{-1},
\end{multline}
and it follows from Lemma \ref{conj} that
$$
R_{\BB H} \bigl( T(q_1), T(q_2), T(q_3), T(q_4) \bigr)
= R_{\BB H}(q_1, q_2, q_3, q_4).
$$

Conversely, suppose that 
$R_{\BB H}(q_1, q_2, q_3, q_4) = R_{\BB H}(q'_1, q'_2, q'_3, q'_4)$.
By Lemma \ref{01infty}, without loss of generality we can assume that
$q_1=q'_1=0$, $q_2=q'_2=1$, $q_3=q'_3=\infty$.
Then by (\ref{q}) we have
$$
|q_4|=|q'_4|, \qquad \re q_4 = \re q'_4.
$$
It follows from Lemmas \ref{conj} and \ref{conj-transf} that there exists
a fractional linear transformation $T: q \mapsto aqa^{-1}$ such that
$T(0)=0$, $T(1)=1$, $T(\infty)=\infty$ and $T(q)=q'$.
\end{proof}

\begin{rem}  \label{igor}
One can generate more invariants of fractional linear transformations as
follows. For an even number of points $q_1,\dots, q_{2n} \in \BB H$, define
$$
\tilde Q(q_1,\dots,q_{2n}) =
(q_1-q_2)(q_2-q_3)^{-1}(q_3-q_4)(q_4-q_5)^{-1} \dots (q_{2n-1}-q_{2n})(q_{2n}-q_1)^{-1}.
$$
Then it follows from Lemma \ref{igor} that
$$
|\tilde Q(q_1,\dots,q_{2n})| =
\bigl|\tilde Q\bigl(T(q_1),\dots,T(q_{2n})\bigr)\bigr|, \qquad
\re \tilde Q(q_1,\dots,q_{2n}) =
\re \tilde Q\bigl(T(q_1),\dots,T(q_{2n})\bigr)
$$
for all fractional linear transformations $T$.
One can even set some of these points equal to each other.
For example, choose $n=3$ and set $q_3=q_6$, then the norm and the real part of
$$
\tilde Q(q_1,q_2,q_3,q_4,q_5) =
(q_1-q_2)(q_2-q_3)^{-1}(q_3-q_4)(q_4-q_5)^{-1}(q_5-q_3)(q_3-q_1)^{-1}.
$$
remain unchanged under fractional linear transformations.
\end{rem}

Note that specifying the values of a fractional linear transformation $T$ at
four points never determines the transformation uniquely.
This is because any four points in $\BB H$ lie on a 2-sphere or a 2-plane.
Any 2-sphere or 2-plane in $\BB H$ can be transformed by a fractional
linear transformation into the unit imaginary sphere
$$
\Theta = \{ q \in \BB H ;\: |q|=1,\: \re q=0 \}.
$$
So, without loss of generality we can assume that $T$ fixes four points on
$\Theta$.
One can show that the set of all fractional linear transformations
fixing $\Theta$ consists of the rotation matrices:
$$
\{ \gamma \in GL(2,\BB H);\: \pi(\gamma)(q)=q \: \forall q \in \Theta \}
= \biggl\{ \gamma_{\theta} = \begin{pmatrix} \cos\theta & -\sin\theta \\
\sin\theta & \cos\theta \end{pmatrix} \in GL(2,\BB H);\:
\theta \in \BB R \biggr\}.
$$
Finally, composing $T$ with $\pi(\gamma_{\theta})$ results in new fractional
linear transformations that have the same values at the selected four points.
In Proposition \ref{unique} we will give conditions that
determine a fractional linear transformation uniquely.

\section{Properties of the Quaternionic Cross-Ratio}

In this section we discuss some properties of the quaternionic
cross-ratio and fractional linear transformations on $\BB H$.
First we fix the images of three points in $\BB H$ and geometrically
characterize all possible images of a fourth point under fractional
linear transformations.

\begin{prop}  \label{4points}
Let $q_1,q_2,q_3, q_4 \in \widehat{\BB H}$ be four distinct points, and
let $q'_1,q'_2, q'_3 \in \widehat{\BB H}$ be three distinct points.
Then the set
\begin{equation}  \label{2sphere}
S = \{ \pi(\gamma)(q_4) ;\: \gamma \in GL(2,\BB H),\:
\pi(\gamma)(q_n)=q'_n,\:n=1,2,3 \} \quad \subset \widehat{\BB H}
\end{equation}
is either a 2-sphere, a 2-plane or a single point.
The degenerate case when this set is a point happens if and only if
$Q(q_1,q_2,q_3,q_4)$ is real.
\end{prop}

\begin{proof}
To simplify notations, we write $Q$ for $Q(q_1,q_2,q_3,q_4)$.
By Theorem \ref{main}, $s \in S$ if and only if
\begin{align}
\bigl| (q'_2 - q'_1)^{-1} (s - q'_1)(s - q'_3)^{-1}(q'_2 - q'_3) \bigr| &= |Q|
\qquad \text{and}  \label{norm}  \\
\re \bigl( (q'_2 - q'_1)^{-1} (s - q'_1)(s - q'_3)^{-1}(q'_2 - q'_3) \bigr)
&= \re Q.  \label{trace}
\end{align}
Expanding (\ref{norm}) and using the properties of the norm gives
$$
|Q|^2 = \frac{|q'_2-q'_3|^2 \cdot |s-q'_1|^2}{|q'_2-q'_1|^2 \cdot |s-q'_3|^2}
\qquad \Longleftrightarrow \qquad
|s - q'_1|^2 = |Q|^2 \frac{|q'_2 - q'_1|^2 \cdot |s - q'_3|^2}{ |q'_2 - q'_3|^2}.
$$
Write
\begin{align*}
s &= t + xi + yj + zk, \\
q'_1 &= a_0 + a_1 i + a_2 j + a_3 k, \\
q'_3 &= b_0 + b_1 i + b_2 j + b_3 k, \\
A &= |Q|^2 \frac{|q'_2 - q'_1|^2}{ |q'_2 - q'_3|^2}.
\end{align*}
Then the above equation becomes
$$
(t - a_0)^2 + (x-a_1)^2 + (y - a_2)^2 + (z - a_3)^2
= A \bigl( (t - b_0)^2 + (x-b_1)^2 + (y - b_2)^2 + (z - b_3)^2 \bigr).
$$
If $A=1$, this is the equation of a 3-plane.
Otherwise, expanding and completing the square for each
variable gives an equation of the form
$$
(t - c_0)^2 + (x-c_1)^2 + (y-c_2)^2 + (z-c_3)^2 = B
$$
for some $c_0, c_1, c_2, c_3, B \in \BB R$.
Since this equation has at least one solution, $B \ge 0$ and
this is the equation of either a 3-sphere in $\BB H$ or a single point.

Consider now the condition (\ref{trace}). We can rewrite it as follows:
$$
\re \bigl( (q'_2 - q'_3)(q'_2 - q'_1)^{-1} (s - q'_1)(s - q'_3)^{-1} \bigr)
= \re Q.
$$
Write
\begin{align*}
\alpha = (q'_2 - q'_3)(q'_2 - q'_1)^{-1}
&= \alpha_0+\alpha_1i+\alpha_2j+\alpha_3k,  \\
\beta = q'_3-q'_1 &= \beta_0+\beta_1i+\beta_2j+\beta_3k, \\
u = s - q'_3 &=  t + xi + yj + zk.
\end{align*}
Then (\ref{trace}) becomes
$$
\re Q = \re\bigl( \alpha (u+\beta)u^{-1} \bigr)
= \re\alpha + |u|^{-2} \re(\alpha\beta \bar u )
= \alpha_0 + |u|^{-2} (\delta_0t + \delta_1x + \delta_2y + \delta_3z)
$$
for some real numbers $\delta_0,\delta_1,\delta_2,\delta_3$
that can be expressed in terms of $\alpha_n$'s and $\beta_m$'s.
Multiplying through by $|u|^2$ gives
$$
(\re Q -\alpha_0)(t^2+x^2+y^2+z^2)
- (\delta_0t + \delta_1x + \delta_2y + \delta_3z) =0.
$$
Once again, if $\re Q -\alpha_0=0$, this is the equation for a 3-plane.
Otherwise, completing the square for each variable gives
\begin{equation*} 
(t-d_0)^2 + (x-d_1)^2 + (y-d_2)^2 + (z-d_3)^2 = D.
\end{equation*}
Since this equation has at least one solution, $D \ge 0$ and
this is the equation of either a 3-sphere in $\BB H$ or a single point.

Thus, the set $S$ can be realized as an intersection of two
sets, each of which is either a 3-sphere, a 3-plane or a point.
Since $S$ is non-empty, this implies that $S$ is either 
a 3-sphere, a 3-plane, a 2-sphere, a 2-plane or a point.
Since fractional linear transformations are diffeomorphisms,
they preserve the dimensions of submanifolds.
Hence, to pin down the dimension of $S$, we can use Lemma \ref{01infty}
and assume without loss of generality that $q_1=q'_1=0$, $q_2=q'_2=1$,
$q_3=q'_3=\infty$.
Then by Lemmas \ref{conj} and \ref{conj-transf}, the set $S$ is either
2-dimensional or a single point, and the latter case happens if and only if
$aQ(q_1,q_2,q_3,q_4)a^{-1} = Q(q_1,q_2,q_3,q_4)$ for all $a \in \BB H^{\times}$,
i.e. if and only if $Q(q_1,q_2,q_3,q_4)$ is real.
\end{proof}

Recall that fractional linear transformations over $\BB C$ map
circles and straight lines into circles and straight lines.
As a consequence of the above proposition we obtain:

\begin{thm}  \label{spheres-thm}
Fractional linear transformations on $\widehat{\BB H}$ send circles and lines
into circles and lines, 2-spheres and 2-planes into 2-spheres and 2-planes,
3-spheres and 3-planes into 3-spheres and 3-planes. 
\end{thm}

\begin{proof}
Consider first a 2-sphere or a 2-plane $S$.
Note that any 2-sphere (respectively 2-plane) in $\BB H$
can be transformed into any other 2-sphere (respectively 2-plane) by
a transformation of the type $q \mapsto aq+b$, $a,b \in \BB H$.
Hence we can assume that our set $S$ can be realized as (\ref{2sphere})
for some choice of $q_1,q_2,q_3,q_4 \in \widehat{\BB H}$.
Now, let $T$ be any fractional linear transformation. Then
$$
T(S) = \{ \pi(\gamma)(q_4) ;\: \gamma \in GL(2,\BB H),\:
\pi(\gamma)(q_n)=T(q'_n),\:n=1,2,3 \},
$$
which by the above proposition is either a 2-sphere or a 2-plane.

Now, let $C \subset \widehat{\BB H}$ be a circle or a line.
Any circle or line can be expressed as an intersection
$C = S_1 \cap S_2$, where $S_1$ and $S_2$ are 2-spheres or 2-planes.
Then $T(C) = T(S_1) \cap T(S_2)$ is the intersection of 2-spheres and/or
2-planes, so is also either a circle or a line.

Finally, let $R$ be a 3-sphere or a 3-plane.
Applying a linear transformation $q \mapsto aq+b$ as above and using the
characterization of the set of points satisfying $|Q(q_1,q_2,q_3,q)|=const$
given in the proof of Proposition \ref{4points}, we can realize $R$ as
$$
R = \{ q \in \BB H ;\: |Q(q_1,q_2,q_3,q)|=N \}
$$
for some choice of $q_1,q_2,q_3 \in \widehat{\BB H}$ and $N \in (0,\infty)$.
Then
$$
T(R) = \bigl\{ q \in \BB H ;\:
\bigl|Q \bigl(T(q_1),T(q_2),T(q_3),q \bigr)\bigr|=N \bigr\},
$$
which is a 3-sphere or a 3-plane as well.
\end{proof}  

Theorem \ref{spheres-thm}  was originally proved in \cite{BG},
but their proof did not use the quaternionic cross-ratio.
Instead, they checked that 3-spheres and 3-planes get mapped into
3-spheres and 3-planes by the generators of the group of fractional linear
transformations $\{ \pi(\gamma) ;\: \gamma \in GL(2,\BB H)\}$, namely the
translations, rotations, dilations and the inversion $q \mapsto q^{-1}$.
Then they realized 2-spheres, 2-planes, circles and lines as finite
intersections of 3-spheres and 3-planes.

Recall that, over complex numbers $\BB C$, the cross-ratio (\ref{C-ratio}) is
real if and only if the four points $z_1,z_2,z_3,z_4 \in \BB C$ lie on a single
circle or a single line. This result also carries over to $\BB H$: 

\begin{prop}  \label{real}
Let $q_1,q_2,q_3,q_4 \in \widehat{\BB H}$.
Then $Q(q_1, q_2, q_3, q_4)$ is real if and only if $q_1,q_2,q_3,q_4$ lie on a
single circle or a single line. 
\end{prop}

\begin{proof}
Let $T$ denote the fractional linear transformation $T_{q_1,q_2,q_3}$ from
Lemma \ref{01infty}. Then, by Theorem \ref{spheres-thm},
$q_1,q_2,q_3,q_4$ lie on a single circle or a line if and only if the points
$0,1,\infty,T(q_4)$ do, i.e. if and only if $T(q_4)$ is real. But, $T(q_4)$ is exactly the quaternion $Q(q_1, q_2, q_3, q_4)$. 
\end{proof}

Recall from Proposition \ref{4points} that the set $S$ defined by
(\ref{2sphere}) degenerates into a single point if and only if the cross-ratio
$Q(q_1,q_2,q_3,q_4)$ is real. Now we know that this happens if and only if
the points $q_1,q_2,q_3,q_4$ lie on single circle or a single line.

\section{Five Points and the Unique Determination of
Fractional Linear Transformations}  \label{5}

As it was explained after Remark \ref{igor}, specifying values of a
fractional linear transformation at four points never determines that
transformation uniquely.
We shall give a necessary and sufficient condition for the existence of a
fractional linear transformation with prescribed values at five different
points and discuss the uniqueness of such a transformation.

\begin{prop}  \label{main2}
Let $q_1, q_2, q_3, q_4, q_5 \in \widehat{\BB H}$ be five distinct points,
and let $q'_1, q'_2, q'_3, q'_4, q'_5 \in \widehat{\BB H}$ be another
collection of five distinct points.
Then there exists a fractional linear transformation $T$ such that
$T(q_n)=q'_n$, $n=1,2,3,4,5$, if and only if there exists an
$a \in \BB H^{\times}$ such that
$$
Q(q'_1,q'_2,q'_3,q'_4) = a Q(q_1,q_2,q_3,q_4) a^{-1}
\qquad \text{and} \qquad
Q(q'_1,q'_2,q'_3,q'_5) = a Q(q_1,q_2,q_3,q_5) a^{-1}.
$$ 
\end{prop}

\begin{proof}
For simplicity write
$$
Q_4 = Q(q_1,q_2,q_3,q_4), \quad Q'_4 = Q(q'_1,q'_2,q'_3,q'_4), \quad
Q_5 = Q(q_1,q_2,q_3,q_5), \quad Q'_5 = Q(q'_1,q'_2,q'_3,q'_5).
$$
By (\ref{Q-conj}), if $T$ is a fractional linear transformation on
$\widehat{\BB H}$, replacing $q_n$ with $T(q_n)$, $n=1,2,3,4,5$,
results in conjugating $Q_4$ and $Q_5$ by the same quaternion.
Thus, by Lemma \ref{01infty}, we can assume without loss of generality that
$q_1=q'_1=0$, $q_2=q'_2=1$, $q_3=q'_3=\infty$.
Then, by (\ref{q}), we have $Q_4=q_4$, $Q'_4=q'_4$, $Q_5=q_5$ and $Q'_5=q'_5$.
On the other hand, by Lemma \ref{conj-transf}, any transformation $T$
such that $T(q_n)=q'_n$, $n=1,2,3$, has to be of the form $q \mapsto aqa^{-1}$
for some $a \in \BB H^{\times}$.
This proves that there exists a fractional linear transformation $T$ such that
$T(q_n)=q'_n$, $n=1,2,3,4,5$, if and only if there exists an
$a \in \BB H^{\times}$ such that $Q'_4=aQ_4a^{-1}$ and $Q'_5=aQ_5a^{-1}$.
\end{proof}

In light of the discussion preceding Lemma \ref{conj}, we can restate
Proposition \ref{main2} as follows.
Given five distinct points $q_1,q_2,q_3,q_4,q_5 \in \widehat{\BB H}$
and five more distinct points $q'_1,q'_2,q'_3,q'_4,q'_5 \in \widehat{\BB H}$,
there exists a fractional linear transformation $T$ such that
$T(q_n)=q'_n$, $n=1,2,3,4,5$, if and only if
\begin{enumerate}
\item
$|Q(q_1,q_2,q_3,q_4)| = |Q(q'_1,q'_2,q'_3,q'_4)|$ and
$|Q(q_1,q_2,q_3,q_5)| = |Q(q'_1,q'_2,q'_3,q'_5)|$;
\item
$\re Q(q_1,q_2,q_3,q_4) = \re Q(q'_1,q'_2,q'_3,q'_4)$ and
$\re Q(q_1,q_2,q_3,q_5) = \re Q(q'_1,q'_2,q'_3,q'_5)$;
\item
There is a single rotation of the 3-space consisting of
imaginary quaternions which takes
$$
\im Q(q_1,q_2,q_3,q_4) \mapsto \im Q(q'_1,q'_2,q'_3,q'_4) \quad \text{and} \quad
\im Q(q_1,q_2,q_3,q_5) \mapsto \im Q(q'_1,q'_2,q'_3,q'_5).
$$
\end{enumerate}
In general, if $p_1, p_2, p'_1, p'_2$ are points in $\BB R^3$, then there
exists a rotation $\rot \in SO(3)$ such that $\rot(p_1)=p'_1$ and
$\rot(p_2)=p'_2$ if and only if $|p_1|=|p'_1|$, $|p_2|=|p'_2|$ and
$|p_1-p_2|=|p'_1-p'_2|$.
Thus, in the presence of the first two conditions, the last condition is
satisfied if and only if
\begin{enumerate}
\item[$3'$.]
$\bigl| Q(q_1,q_2,q_3,q_4) - Q(q_1,q_2,q_3,q_5) \bigr|
= \bigl| Q(q'_1,q'_2,q'_3,q'_4) - Q(q'_1,q'_2,q'_3,q'_5) \bigr|$.
\end{enumerate}

Next we address the question of uniqueness of
fractional linear transformations on $\widehat{\BB H}$.

\begin{prop}  \label{unique}
Let $q_1,q_2,q_3,q_4,q_5 \in \widehat{\BB H}$ be five distinct points not
lying on a single 2-sphere or 2-plane. Then any fractional linear
transformation $T$ is uniquely determined by its values at these points,
$T(q_n)$, $n=1,2,3,4,5$.
\end{prop}

\begin{proof}
By Lemma \ref{01infty} and Theorem \ref{spheres-thm}, without loss of
generality we can assume that
$q_1=T(q_1)=0$, $q_2=T(q_2)=1$ and $q_3=T(q_3)=\infty$.
Then by, Lemma \ref{conj-transf}, $T$ has to be of the form $q \mapsto aqa^{-1}$
for some $a \in \BB H^{\times}$. Hence $T$ preserves the real parts and is
effectively a rotation of the imaginary 3-space.
Since the points $q_1,q_2,q_3,q_4,q_5$ do not lie
on a single 2-plane, $\im q_4 \ne 0$, $\im q_5 \ne 0$ and
$\im q_5$ is not a real multiple of $\im q_4$.
Therefore, $\im q_4$ and $\im q_5$ do not lie on a single line through
the origin, and the rotation $T$ is uniquely determined by its values on
$q_4$ and $q_5$.
\end{proof}

To determine if five points lie on the same 2-sphere or a 2-plane
we can use the cross-ratio:

\begin{lem}  \label{sphere-crit}
Five different points $q_1,q_2,q_3,q_4,q_5 \in \widetilde{\BB H}$ lie on
a single 2-sphere or a 2-plane if and only if
$Q(q_1,q_2,q_3,q_4)$ and $Q(q_1,q_2,q_3,q_5)$ commute with each other.
\end{lem}

\begin{proof}
As in the proof of Proposition \ref{main2}, let
$Q_4 = Q(q_1,q_2,q_3,q_4)$ and $Q_5 = Q(q_1,q_2,q_3,q_5)$.
By (\ref{Q-conj}), if $T$ is a fractional linear transformation on
$\widehat{\BB H}$, replacing $q_n$ with $T(q_n)$, $n=1,2,3,4,5$,
results in conjugating $Q_4$ and $Q_5$ by the same quaternion.
Thus by Lemma \ref{01infty} and Theorem \ref{spheres-thm},
we can assume without loss of generality that $q_1=0$, $q_2=1$, $q_3=\infty$.
By (\ref{q}), we have $Q_4=q_4$ and $Q_5=q_5$.
Then the points $q_1,q_2,q_3,q_4,q_5$ lie on a single 2-plane if and only if
one of $\im Q_4$, $\im Q_5$ is a real multiple of the other.
By Lemma \ref{commute}, this is equivalent to $Q_4$ and $Q_5$
commuting with each other.
\end{proof}

Finally, we fix the images of four points in $\BB H$ and characterize
all possible images of a fifth point under fractional linear transformations,
just like we did in Proposition \ref{4points} with four points.

\begin{prop}  \label{5points}
Let $q_1,q_2,q_3,q_4,q_5 \in \widehat{\BB H}$ be five distinct points, and
let $q'_1,q'_2,q'_3,q'_4 \in \widehat{\BB H}$ be four distinct points.
Assume that
\begin{equation}  \label{assumption}
R_{\BB H}(q_1, q_2, q_3, q_4) = R_{\BB H}(q'_1, q'_2, q'_3, q'_4)
\end{equation}
and that $\im Q(q_1, q_2, q_3, q_4) \ne 0$ (or, equivalently,
that the points $q_1,q_2,q_3,q_4$ do not lie on a single circle or a line).
Then the set
\begin{equation}
C = \{ \pi(\gamma)(q_5) ;\: \gamma \in GL(2,\BB H),\:
\pi(\gamma)(q_n)=q'_n,\:n=1,2,3,4 \} \quad \subset \widehat{\BB H}
\end{equation}
is either a circle, a line or a single point.
The degenerate case when this set is a point happens if and only if
$Q(q_1,q_2,q_3,q_4)$ and $Q(q_1,q_2,q_3,q_5)$ commute with each other.
\end{prop}

\begin{proof}
By Lemma \ref{01infty} and Theorem \ref{spheres-thm}, without loss of
generality we can assume that $q_1=q'_1=0$, $q_2=q'_2=1$, $q_3=q'_3=\infty$.
Then, by (\ref{q}),
$$
Q(q_1,q_2,q_3,q_4)=q_4 \qquad \text{and} \qquad Q(q_1,q_2,q_3,q_5)=q_5;
$$
by Lemma \ref{conj-transf},
$$
C= \{ aq_5a^{-1} ;\: a\in \BB H^{\times},\: aq_4a^{-1}=q'_4 \};
$$
by Theorem \ref{main} and assumption (\ref{assumption}),
the set $C$ is non-empty.

Geometrically, following the discussion preceding Lemma \ref{conj},
we are essentially looking at the set of all $\rot(\im q_5)$,
where $\rot \in SO(3)$ runs over all rotations taking $\im q_4$ into $\im q'_4$.
But all such rotations have the form $\rot' \circ \rot_0$,
where $\rot_0 \in SO(3)$ is a fixed rotation such that
$\rot_0(\im q_4) = \im q'_4$ and $\rot' \in SO(3)$ is a rotation about
the line passing through $\im q'_4$.
This proves that the set $C$ is either a circle, a line or a single point.

The case when $C$ is a point happens if and only if the rotation $\rot_0$
takes $\im q_5$ into a point lying on the line passing through $\im q'_4$.
That happens if and only if $\im q_5$ is a real multiple of $\im q_4$,
which, by Lemma \ref{commute}, happens if and only if
$Q(q_1,q_2,q_3,q_4)$ and $Q(q_1,q_2,q_3,q_5)$ commute.
\end{proof}

\end{document}